\documentclass{amsart}
\usepackage{amssymb,amsmath,amsthm}
\usepackage{hyperref}
\usepackage{mathabx}
\usepackage{mathtools}
\usepackage[all]{xy}
\usepackage[T1]{fontenc}
\usepackage[utf8]{inputenc}
\usepackage{lmodern}
\usepackage [american]{babel}

\usepackage{enumitem}

\newcommand{\Kl}[1]{\left( #1 \right)}
\newcommand{\legendre}[2]{\left( \frac{#1}{#2} \right)}
\newcommand{\conj}[1]{\bar{#1}}

\newcommand{\Z}{{\mathbb Z}}
\newcommand{\N}{{\mathbb N}}
\newcommand{\Q}{{\mathbb Q}}
\newcommand{\R}{{\mathbb R}}
\newcommand{\C}{{\mathbb C}}
\newcommand{\Ha}{{\mathbb H}}

\newcommand{\be}{{\beta}}
\newcommand{\al}{{\alpha}}
\newcommand{\eps}{{\epsilon}}

\newcommand{\A}{{\mathfrak a}}

\def\ysum{\mathop{\sum \!{\vphantom{\sum}}^\prime }}
\newcommand{\psum}[2]{\ysum_{{#1}(\text{mod\,}{#2)}}}

\newtheorem{thm}{Theorem}
\newtheorem{lemma}[thm]{Lemma}

\theoremstyle{definition}

\DeclarePairedDelimiter\abs{\lvert}{\rvert}
\DeclarePairedDelimiter\norm{\lVert}{\rVert}

\makeatletter
\let\oldabs\abs
\def\abs{\@ifstar{\oldabs}{\oldabs*}}

\let\oldnorm\norm
\def\norm{\@ifstar{\oldnorm}{\oldnorm*}}
\makeatother

\title[Fourier coefficients of half-integral weight cusp forms]{Fourier coefficients of half-integral weight cusp forms and Waring's problem}
\author{Fabian Waibel}

\begin{document}
\begin{abstract} 
Extending the approach of Iwaniec and Duke, we present strong uniform bounds for Fourier coefficients of half-integral weight cusp forms of level $N$. As an application, we consider a Waring-type problem with sums of mixed powers. 
\end{abstract} 
\maketitle
\section{Introduction}

A positive, integral, symmetric $k \times k$ matrix $A$ with even diagonal elements gives rise to a quadratic form $q(x):= \frac{1}{2} x^{t} A x$. It is a central problem of number theory to study the representation function 
\begin{align} \notag
r(q,n) := \# \{ x \in \Z^{k} \, | \, q(x) = n \}.
\end{align}
One way to do so is by examining the theta series 
\begin{align} \notag
\theta(q,z) := \sum_{x \in \Z^{k}} e(q(x)z) = \sum_{n = 0}^{\infty} r(q,n) e(nz)
\end{align}
which is a modular form of (generally) half-integral weight of level $N$, where $N$ is the level of $q$. To understand $r(q,n)$, decompose $\theta(q,z)$  into an Eisenstein series and a cusp form. To treat the cusp form contribution,  one may apply the results from the late eighties from Iwaniec \cite{Iw1}, Duke \cite{Du}, and Duke-Schulze-Pillot \cite{DS}. Let $f(z) = \sum_{n \geq 1} a(n)e(nz)$ be a holomorphic cusp form of half-integral weight  $k/2,\,k\geq 3$ for the group $\Gamma = \Gamma_0(N)$, normalized with respect to 
\begin{align}  \label{eq:1.0}
\langle f,g \rangle = \int_{\Gamma \backslash \Ha} f(z) \overline{g(z)}  y^{k} \frac{dx dy}{y^2}.
\end{align}
Then, it was shown in \cite{Du,Iw1} that, for squarefree $n$,  
\begin{align} \label{eq:2.0}
a(n) \ll n^{k/4-2/7+\eps}
\end{align}
provided that $f \in U^{\bot}$ for $k=3$, where $U$ is the subspace of theta functions of $S_{3/2}(N,\chi)$ of type $\sum_{n \geq 1} \psi(n) n e(tn^2z)$ for some real character $\psi$ and $4t \divides N$. 

The aim of this paper is threefold: we extend the bound \eqref{eq:2.0} to arbitrary $n$, we include forms of level $N$ with arbitrary nebentypus and improve the bound with respect to $N$. For the second point we need to  bear in mind that the Weil-Estermann bound does not necessarily hold for twisted Kloosterman sums for prime power moduli  
(cf. \cite[Example 9.9]{KL}). The main strategy follows the work of Duke and Iwaniec \cite{Du,Iw1} with the extensions of Blomer \cite{Bl1}. 

If $d$ divides a power of $x$, we write $d \divides x^{\infty}$, and we denote the squarefree kernel by $\textnormal{rad}(n)$. 

\begin{thm} \label{main1}
Fix an orthonormal basis $\{\varphi_{j}=\sum_{n\geq1} a_{j}(n) e(nz)\}_{j=1}^{d}$ of $S_{k/2}(N,\chi)$ for odd $k \geq 5$ and of $U^{\bot}$ for $k=3$. Then it holds for $n= t v^2 w^2$ with $t$ squarefree, $v \divides N^{\infty}$, $(w,N)=1$  and quadratic $\chi$  that %$\abs{a(tv^2w^2)} \leq \abs{a(tw^2} v^{k-1+\eps}$ and thus
\begin{align}  \notag
\sum_{j=1}^{d} \abs{a_j(n)}^2 \ll n^{k/2-1} \Kl{\frac{t^{3/7}v^{6/7}}{N^{2/7}(n,N)^{1/7}} + \frac{t^{3/8}v^{3/4}}{N^{1/8}(n,N)^{1/4}}+  \frac{v(n,N)}{N}+1} (nN)^{\eps}.
\end{align} 
For arbitrary $\chi$, the last term within the bracket changes to $\frac{v (n,N)}{N } \Kl{c_{\chi} \textnormal{rad}(c_\chi)}^{1/4}$, where $c_\chi$ is the conductor of $\chi$. 
\end{thm}
%Theorem \ref{main1} is ultimately based on Iwaniec's method of bounding sums of half-integral weight Kloosterman sums. By a very different approach, based on Waldspurger's theorem and subconvexity, one can bound $a_j(n)$ by $\mathcal{O}(n^{k/4-5/16})$  cf. \cite[Corollary 2]{Bh} which is slightly better in terms of $n$. However, the implied constant depends quite strongly on $N$. For many applications involving families of cusp forms, such as the one presented below, Theorem \ref{main1} leads therefore to stronger results. 

We singled out the case of quadratic $\chi$ because this is the relevant case for quadratic forms and the main application that we proceed to present. It has been investigated by Wooley \cite{Wo} under which conditions on the exponents $k_j,\, j=1,\ldots,t$ the Diophantine equation 
\begin{align}  \label{eq:Wo}
x_{1}^2 + x_{2}^2+ x_{3}^3 + x_{4}^3 + \sum_{j=1}^{t} y_{j}^{k_j} = n
\end{align}
has solutions for all sufficiently large $n$. His proof is based, among other things, on a  result of Golubeva \cite[Theorem 2]{Go} which we can improve by  Theorem \ref{main1} as follows.
\begin{thm}  \label{Golubeva}
Let $P\neq 3$ be an odd prime, $(n,6P)=1$ and $n=tv^2$ with $t$ squarefree. Then 
\begin{align} \notag
n = x^2 + y^2 +6 P z^2
\end{align}
is solvable for $(x,y,z) \in \N^{3}$ if $P^{1+\eps} \leq \rm{min}(n^{1/17}v^{12/17},n^{1/11}v^{6/11},n^{1/3})$. This holds, in particular, if $nv^{28/3}>P^{17 + \eps}$. 
\end{thm}
In \cite{Go}, the  bound is $nv^{12} > P^{21+\eps}$. For $k_i \in \N$ and $2 \leq k_1 \leq \ldots \leq k_t$ set 
\begin{align} \notag
\gamma(k) = \prod_{i=1}^{t} \Kl{1- \frac{1}{k_i} } \quad \text{and} \quad \tilde{\gamma}(k) = \Kl{1- \frac{1}{k_t} }  \prod_{i=1}^{t-2} \Kl{1- \frac{1}{k_i} }.
\end{align}
\begin{thm} \label{Wo} Assume the Riemann hypothesis for all $L$-functions associated with Dirichlet characters. Then, provided that $\gamma(k) < \frac{28}{39}$, all sufficiently large numbers $n$ are represented in the form of \eqref{eq:Wo}. The same conclusions hold without the assumption of the Riemann hypothesis if 
\begin{enumerate}[label=(\roman*)]
\item  $t \geq 2$ and $\tilde{\gamma}(k) < \frac{58}{81}$  or
\item $\gamma(k) < \frac{58}{81}$ and the exponents $k_1,\ldots,k_t$ are not all even.
\end{enumerate}
\end{thm}
The original bounds in \cite[Theorem 1.2]{Wo} are $\gamma(k) < 12/17$ with the assumption of the Riemann hypothesis, $\tilde{\gamma}(k) < 74/105$ for $(i)$ and $\gamma(k) < 74/105$ for $(ii)$. As a consequence, it follows that  every sufficiently large number $n$ is represented in the form
\begin{align} \notag
x_{1}^2+ x_{2}^2 + x_{3}^3 + x_{4}^3 + \sum_{j=1}^{t} x_{j}^{3t}  = n,
\end{align}
with odd $t \leq 81$, or in the form 
\begin{align} \notag
x_{1}^2+ x_{2}^2 + x_{3}^3 + x_{4}^3 + x_{5}^{8}+ x_{6}^{12}+x_{7}^{16}+x_8^{20} =n
\end{align} 
if the truth of the Riemann hypothesis is assumed. 
 
\textbf{Acknowledgments.} I would like to express my gratitude to Prof. Blomer for the many useful suggestions and remarks.

\section{Shimura's lift and Maa{\ss} forms}

We follow the exposition of \cite{Bl1}. For $0 \neq z \in \C$ and $v \in \R$  define $z^v$ by
\begin{align}  \notag
z^v = \abs{z}^v \text{exp}(iv  \text{ arg}(z)), \text{ where } \text{arg}(z) \in (-\pi,\pi]. 
\end{align}  
For a holomorphic function on the upper half plane $f: \Ha \to \C$ and $\gamma = \begin{pmatrix} a &  b \\ c  & d \end{pmatrix} \in \Gamma_0(4)$ set 
\begin{align}  \notag
f|[\gamma]_{k/2}(z) =  \Kl{\eps_{d}^{-1} \legendre{c}{d}}^{-k} (cz+d)^{-k/2} f(\gamma z), 
\end{align} 
where $\legendre{c}{d}$ is the extended Kronecker symbol (cf. \cite[p.\,442]{Sh}) and $\eps_d = \Kl{\frac{-1}{d}}^{1/2}$. From now on, $\chi$ will always denote a character mod $N$ and $4 \divides N$. For odd $k$, we denote the spaces of  modular forms and cusp forms of half-integral weight $k/2$ for $\Gamma_0(N)$  and transformation behavior $f|[\gamma]_{k/2}(z) = \chi(d) f(z)$ by $M_{k/2}(N,\chi)$ and $S_{k/2}(N,\chi)$. For $f,g \in S_{k/2}(\Gamma, \chi) $, the inner product is defined by  \eqref{eq:1.0}. For $(n,N) =1$, let $T(n): M_k(N,\chi) \to M_k(N,\chi)$ be the Hecke operator (cf. \cite[Chapter\,4.3]{Ko}). 

For $f=\sum_{n \geq 1} c(n)e(nz) \in S_{k/2}(N,\chi)$, $k \geq 3$ odd, $\varepsilon=(-1)^{(k-1)/2}$ and $t$ without square factors (other than 1) prime to $N$,  define $C_t(n)$  by the formal identity 
\begin{align} \notag
\sum_{n=1}^{\infty} C_t(n) n^{-s} = L(s-k/2+3/2,\chi_{4 \varepsilon t}\chi)\sum_{n=1}^{\infty} c(t n^2) n^{-s}.
\end{align}
Then $F_t(z) = \sum_{n=1}^{\infty} C_t(n) e(nz) \in M_{k-1}(N/2,\chi^{2})$ is called the $t$-Shimura lift. If $f$ is an eigenform for all Hecke operators $T(p^2),~p \nmid N$ with eigenvalues $\lambda_p$, then $F_t$, if it is not equal to $0$, is an eigenform for all $T_p, p \nmid N$ with the same eigenvalues, and it holds for $(n,N)=1$ that \cite[Corollary 1.8]{Sh} 
\begin{align}  \notag
C_t(n)= c(t) \cdot \lambda_n.
\end{align}
There exists an orthonormal basis of $U^{\bot}$ and of $S_{k/2}(N,\chi), k \geq 5$,   of simultaneous eigenforms for all $T(p^2), p \nmid N$. Consequently, if the $t$-Shimura lift of $f$ is cuspidal, it follows by Deligne's bound for integral-weight modular forms for $(w,N)=1$ that
\begin{align}  \label{eq:5}
\big| c (t w^2) \big|  = \Big| c(t) \sum_{d \divides w} \mu(d) \chi_{4 \eps t}\chi(d) d^{k/2-3/2} \lambda_{w/m} \Big| \leq \abs{c(t)} w^{k/2-1} \tau(w)^2. 
\end{align}
For $k \geq 5$ the Shimura lift is always cuspidal. However, for $k=3$ the $t$-Shimura lift is cuspidal for all squarefree $t$ if and only if $f \in U^{\bot}$, i.e. $f$ does not live in the subspace of theta functions. 

The theory of  Maa{\ss} forms with general weights was introduced by Selberg \cite{Se}. For $\gamma \in \Gamma_0(4)$ and $k \in \Z$ set 
\begin{align} \notag
f|[\gamma]_{k/2} (z) =  \Kl{\eps_{d}^{-1} \legendre{c}{d}}^{-k} e^{-i (k/2) \text{arg}(cz+d)}  f(\gamma z).
\end{align}
We call a function $f:~ \Ha \to \C$ an automorphic form of weight $k/2$ if it satisfies for  all $
\gamma  \in \Gamma$ the transformation rule 
\begin{align} \notag
f|[\gamma]_{k/2} (z) = \chi(d) f(z),  
\end{align}
and  $f(z) \ll y^{\sigma} + y^{1-\sigma}$ for some $\sigma >0$. A \textit{Maa{\ss} form}  is an automorphic form that is an eigenfunction of 
\begin{align} \notag
\Delta_{k/2} = - y^2 \Kl{ \frac{\partial^2}{\partial x^2} + \frac{\partial^2}{\partial y^2} } + i (k/2) y \frac{\partial}{\partial x},
\end{align}
with eigenvalue  $\lambda = s (1-s)$. We denote the space of such forms by $\mathcal{A}_s(\Gamma \setminus \Ha,k/2,\chi).$ Their inner product is defined by 
\begin{align}  \label{eq:4}
\langle f,g \rangle = \int_{\Gamma \backslash \Ha} f(z) \overline{g(z)} \frac{dx dy}{y^2}.
\end{align}
Every form $f$ in  $\mathcal{A}_s(\Gamma \setminus \Ha,k/2,\chi)$ has a Fourier expansion at the cusp $\infty$ given by
\begin{align}  \label{eq:2.2}
f(z) &=\rho^{+} y^s +  \rho^{-} y^{s-1} +  \sum_{n \in \Z, n \neq 0}
  \rho(n) W_{ \text{sgn}(n)k/4 ,s-1/2} (4 \pi \abs{n}y)e(nx),
\end{align}
where  $W_{\alpha,\beta}(z)$ denotes the standard Whittaker function \cite[p.\,295]{Ob}.  If the zero coefficient of $f \in \mathcal{A}_s(\Gamma \setminus \Ha,k/2,\chi) $ vanishes at every cusp, then it is called a \textit{Maa{\ss} cusp form} and the space of such forms is denoted by $\mathcal{C}_s(\Gamma \setminus \Ha,k/2,\chi)$.

\section{Proof of Theorem 1}

Let $\{\varphi_{j}=\sum_{n\geq1} a_{j}(n) e(nz)\}_{j=1}^{d}$ be an orthonormal basis of $S_{k/2}(N,\chi)$ for odd $k \geq 5$ and of $U^{\bot}$ for $k=3$. Set $n=tv^2w^2$ with $\mu^2(t)=1$, $v \divides N^{\infty}$ and $(w,N)=1$. The square part of $n$ coprime to $N$, $w$, can be easily handled by \eqref{eq:5} since  $\abs{a_j(n)}^2 \leq w^{k/2-1} \abs{a_j(tv^2)}^2$. Therefore, it is sufficient to prove that
\begin{align}  \notag
\sum_{j=1}^{d} \abs{a_j(n)}^2 \ll n^{k/2-1} \Kl{\frac{t^{3/7}v^{6/7}}{N^{2/7}(n,N)^{1/7}} + \frac{t^{3/8}v^{3/4}}{N^{1/8}(n,N)^{1/4}}+  \frac{v(n,N)}{N}+1} (nN)^{\eps}
\end{align}
for $n=tv^2$, with $\mu^2(t)=1$ and $v$ arbitrary. 

The proof  follows the Iwaniec-Duke approach very closely and we assume some familiarity with the article \cite{Iw1}. For $k\geq 5$, we directly apply the Petersson formula while for $k=3$, we first embed the weight $3/2$ cusp forms into the space of Maa{\ss}  cusp forms of weight $3/2$ via $f(x+iy) \mapsto y^{3/4}f(x+iy)$ and then apply the Kuznetsov formula. The Petersson formula for half-integral weights states that \cite[p.\,89]{Sar}
\begin{align} \notag
\frac{\Gamma(k/2-1)}{(4 \pi n)^{k/2-1} } \sum_{j=1}^{d} \abs{a_j(n)}^2 = 1+2 \pi i^{-k/2} \sum_{N \divides c} c^{-1} J_{k-1}\Kl{\frac{4 \pi n}{c}} K_{\chi}^{k}(n,n,;c),
\end{align}
where $J_{k/2-1}$ is the Bessel function of order $k/2-1$ and 
\begin{align} \label{eq:09}
K_{\chi}^{k}(m,n;c) = \psum{d}{c} \eps_d^{-k} \chi(d) \legendre{c}{d} e\Kl{\frac{md + n\conj{d}}{c}}
\end{align} 
is a twisted Kloosterman sum. If $f(z)$ is a normalized cusp form for $\Gamma_0(N)$ with respect to \eqref{eq:1.0}, then $[\Gamma_0(Q):\Gamma_0(N)]^{-1/2}f(z)$ is a normalized cusp for $\Gamma_0(Q)$ provided that $N \divides Q$. Instead of applying the Petersson formula for the level $N$, we use it for  higher levels $Q= pN$ with primes $p \in \mathcal{P}=\{p\, | \, P< p \leq 2P, p \nmid 2nN \}$. Since $[\Gamma_0(pN):\Gamma_0(N)] \leq p+1$, this yields (cf. \cite[p. 400]{Iw1})
\begin{align}  \label{eq:008}
 \sum_{j=1}^{d} \abs{a_j(n)}^2 \ll n^{k/2-1} \Bigg( P + \sum_{p \in \mathcal{P}} \Bigg| \sum_{(pN) \divides c} c^{-1} K_{\chi}^{k}(n,n;c) J_{k/2-1} \Kl{\frac{4 \pi n}{c} } \Bigg| \text{log}P \Bigg),
\end{align}
where we choose $P > 1+ (\text{log} 2nN)^2$ to ensure that $\# \mathcal{P} \asymp P (\text{log}P)^{-1}$. After expressing the Bessel function by means of its asymptotic formula and applying partial summation, it remains to find a bound for sums of the type $\sum_{Q \in \Q} |K_Q(x)|$, where
\begin{align}  \label{eq:08}
K_Q(x) := \sum_{c \leq x,\, Q \divides c} c^{-1/2} K_{\chi}^{k}(m,n;c) e\Kl{\frac{2\nu n}{c}}
\end{align}
with $-1 \leq \nu \leq 1$ and $Q \in \mathcal{Q} = \{pN \, | \, p \in \mathcal{P} \}$. 

First, we factor the modulus $c$ into $qr$, where $q$ is coprime to $2nN$ and $r \divides (2nN)^{\infty}$. This way, \eqref{eq:09} decomposes into a Kloosterman sum of modulus $r$ and a Salié sum of modulus $q$ which is explicitly computable. Very similar to \cite[Lemma 6]{Iw1}, we obtain
\begin{align}  \label{eq:07}
K_{\chi}^{k}(n,n;c) = &q^{1/2} \sum_{\substack{s (\textrm{mod } r/2)\\ 2 \nmid s}} \eps_s^{-2k} f_r(2s,\chi) \left[(1+i^s)\legendre{nr}{q} + (1-i^s)\legendre{-nr}{q} \right]  \\ \notag
&\sum_{ab=q} e\Kl{2n \Kl{\frac{\overline{ar}}{b}- \frac{\overline{br}}{a} + \frac{s \overline{ab}}{r}}}.
\end{align}
The main difference is that 
\begin{align} \notag
f_r(2s,\chi) = \sum_{\substack{d (\textrm{mod } r)\\ d+\conj{d}\equiv 2s (\text{mod}\,r)}} \legendre{r}{d} \chi(d). 
\end{align}

\begin{lemma} \label{Lemma2.2}
For quadratic $\chi$, one has the following bound
\begin{align} \notag
\abs{K_{\chi}^{k}(n,n;c)} \leq  \tau(c) (n,c)^{1/2} c^{1/2},
\end{align} 
while, for arbitrary  $\chi$ one gets an additional factor of $(c_\chi \textnormal{rad}(c_\chi))^{1/4}$ on the right-hand side. 
\end{lemma}
\begin{proof} 
If we split the sum for $c=r q, r \divides 2^{\infty}, (2,q)=1$ we obtain 
\begin{align}  \label{eq:9}
K_{\chi}^{k}(n,n;c) = K_{\chi_r}^{k-q+1}(n\bar{q},n\bar{q};r)  S_{\chi_{q}} \Kl{ n \bar{r},   n \bar{r} ,q },
\end{align}
where $\chi_r$ and $\chi_q$ are characters modulo $r$ and $q$ respectively and the latter sum 
\begin{align} \notag
S_{\chi}(n,n;q)= \psum{d}{q} \chi(d) \Kl{\frac{d}{q}} e \Kl{\frac{n (d+ \conj{d})}{q}}
\end{align}
is a Kloosterman sum twisted by a character. For arbitrary $\chi$, we apply \cite[Theorem 9.3]{KL} and get $\abs{S_{\chi}(n,n;q)} \leq  \tau(q) (n,q)^{1/2} q^{1/2} (q_\chi \textnormal{rad}(q_\chi))^{1/4}$. Since the conductor of a real character with odd modulus is always squarefree, we obtain the Weil bound for real $\chi$ by applying \cite[Proposition 9.4, 9.7 and 9.8]{KL}, i.e. $\abs{S_{\chi}(n,n;q)} \leq \tau(q ) (n,q)^{1/2} q^{1/2}$. To bound the first term on the right-hand side of \eqref{eq:9}, we modify \cite[Lemma 12.2 and Lemma 12.3]{Iw2}. Therefore, we set $r=2^{\alpha}$ and assume that $\al \geq 4$ to ensure that $\eps_r = \eps_a$ for $r = a + b\, 2^{\be}$, where $\N \ni \be = \frac{\al}{2}$ or $\frac{\al-1}{2}$ respectively. By following the argument of Iwaniec very closely, we obtain
\begin{align} \notag
|K_{\chi}^{k}(n,n;2^{\al})| \leq 2^{\be} M,
\end{align}
where $M$ is the number of solutions modulo $2^{\be}$ of %a quadratic congruence mod $2^{\be}$, namely 
$-n a^2 + Ba +n \equiv 0~ \text{mod}\, 2^{\be}$ for $B$ defined as in \cite[Lemma 12.2 and Lemma 12.3]{Iw2}. To bound $M$, we proceed as in \cite[Lemma 9.6, Proposition 9.7 and Proposition 9.8]{KL} obtaining $|K_{\chi}^{k}(n,n;r)| \leq \tau(r) (n,r)^{1/2}r^{1/2}$. 
\end{proof}
We split $K_Q(x)$ according to whether $t \divides c$. By applying Lemma \ref{Lemma2.2} we get, for quadratic $\chi$, that 
\begin{align} \notag
\abs{K_{[t,Q]}(x)} &\ll    \frac{x(t,Q)(n,[t,Q])^{1/2}}{tQ} \tau(tQ) (xn)^{\eps} \\ \notag
&\leq   \frac{x(t,Q)(v^{2},Q/(Q,t))^{1/2}}{t^{1/2}Q} \tau(tQ) (xn)^{\eps}  \leq  \frac{x v (n,Q)}{n^{1/2}Q} \tau(tQ) (xn)^{\eps}
\end{align} 
since $(t,Q)^{2} (v^{2},Q/(Q,t))$ divides both $Q^{2}$ and $n^{2}$. In particular, one has
\begin{align}  \label{eq:1}
\sum_{Q \in \mathcal{Q}} \abs{K_{[t,Q]}(x)} \ll xv(n,N) n^{-1/2}N^{-1} (xnN)^{\eps}.
\end{align}
For general $\chi$, we get an additional factor of  $(c_{\chi}\text{rad}(c_{\chi}))^{1/4}$ on the right-hand side. The remaining part of $K_Q(x)$ can be reduced to partial sums of the type
\begin{align} \notag 
K_Q^{\star}(y) = \sum_{\substack{y < c \leq 2y \\ t \nmid c, \,Q \divides c}} c^{-1/2} K_{\chi}^{k}(n,n;c)e\Kl{\frac{2 \nu n}{c}}
\end{align}
with $4 \leq y \leq x$. There are $\mathcal{O}(\textnormal{log}(x))$ such partial sums. For even $t$, we trivially estimate $|K_{[t/2,Q]}(x)|$ and assume that $K_Q^{\star}(y)$ runs over $c$ with $\frac{t}{2} \nmid c$ to ensure that  $n/(n,r)$ is not a perfect square. By \eqref{eq:07} we conclude that 
\begin{align}  \label{3.40}
K_Q^{\star}(y) = \sum_{r \in \mathfrak{R}} r^{-1/2} \sum_{\substack{s (\textrm{mod } r/2)\\ 2 \nmid s}} \eps_s^{-k} f_r(2s) \left[(1+i^s) F_{r,s}^{+}(p) + (1-i^s)F_{r,s}^{-}(p) \right] ,
\end{align}
where $\mathfrak{R} = \{r\, ; \,  N \divides r \divides (2nN)^{\infty} ,\,  t \nmid r \}$ and 
\begin{align}  \label{eq:3.3}
F_{r,s}^{\pm}(p) = \mathop{\sum \sum}_{\substack{y < abr \leq 2y \\ (a,b)=1, p \divides ab}} \legendre{\pm nr}{ab}e\Kl{2n \Kl{\frac{\overline{ar}}{b}- \frac{\overline{br}}{a} + \frac{s \overline{ab}}{r} + \frac{\nu}{abr}}}
\end{align}
with $(ab,2nN)=1$. We treat $F_{r,s}^{\pm}(p)$ according to the values of $a$ and $b$ and  split it into dyadic ranges  $A < a \leq 2A$ and $B < b \leq 2B$ with $y < rAB \leq 2y$ and $A,B \geq \frac{1}{2}$ which we denote by $F(A,B;p)$. 

For either $A$ or $B$ small, we apply the Weil bound for the Kloosterman sum and estimate trivially. Following \cite[p.396]{Iw1} word by word, we get 
\begin{align}   \label{eq:3.5}
F(A,B;p) \ll \Kl{1+\frac{n}{y}} \sum_{\substack{B < b \leq 2B \\ (b,2nN)=1}} \bigg| \sum_{\substack{A_1 < a \leq A_2 \\ (a,b)=1} } \legendre{\pm nr}{a} e\Kl{2 n m \frac{\conj{a}}{br} } \bigg|,
\end{align}
with $m$ defined by $m p_b \equiv r \bar{r} + 1 + s b \bar{b} (\textnormal{mod\,}br)$ and $A_1,A_2$ such that $Ap_b  = A_1 < A_2 \leq 2 A p_b$, where $p_b := p / (b,p)$. Set $\delta_1 = \frac{n}{(n,r)}$ and $\delta_2 = \frac{r}{(n,r)}$. At this point, we cannot proceed as in Iwaniec \cite[Section 5]{Iw1} because $8 \divides \delta_2$ is generally not satisfied. To solve this, we distinguish three cases:
\begin{itemize}
\item $ 2 \nmid \delta_1$. Set $\Delta_1 = \delta_1$ and $\Delta_2 = 16 \delta_2$.
\item $\text{ord}_2(\delta_1)=1$ or $2$. Set $\Delta_1 = 2^{-\text{ord}_2(\delta_1)} \delta_1 $ and $\Delta_2 = 2^{2+\text{ord}_2(\delta_1)}\delta_2$. 
\item $8 \divides \delta_1 $. Set $\Delta_1 = \delta_1$ and $\Delta_2 = \delta_2$. 
\end{itemize}
In each case $\Delta_1$ and $\Delta_2$ satisfy that  $\legendre{ \pm nr}{a} = \legendre{\pm \Delta_1 \Delta_2}{a}$, either $8 \divides \Delta_1$ or $8 \divides \Delta_2$ and $\Delta_1, \Delta_2$ and $b$ are pairwise coprime. Set $2 \frac{n}{r} = 2^{j} \frac{\Delta_1}{\Delta_2}$, where $j=5$, $j=3 + 2 \, \text{ord}_2(\delta_1)$ or $j=1$ according to the corresponding case. Thus, the innermost sum of \eqref{eq:3.5} is equal to 
\begin{align} \notag
\sum_{a} := \sum_{A_1 < a \leq A_2 } \legendre{\pm \Delta_1 \Delta_2}{a} e \Kl{2^{j} m\frac{\Delta_1 \conj{a}}{\Delta_2 b} }.
\end{align}
By applying \cite[$(3.14)$]{Iw1} it follows for $D=\Delta_1 \Delta_2 b$ that
\begin{align} \notag
\bigg| \sum_{a} \bigg|  \leq \sum_{1 \leq \abs{d} \leq D/2} \frac{1}{2 \abs{d}} \Bigg| \sum_{x (\text{mod}D)} \legendre{\pm \Delta_1 \Delta_2}{x} e \Kl{2^{j} m\frac{\Delta_1 \conj{x}}{\Delta_2 b} + \frac{dx}{D}} \Bigg|.
\end{align}
The sum modulo $D$ can be factored into three sums in the same manner as in \cite[p.\,396]{Iw1}. Note that $\Delta_1$ is not a perfect square because there exists an odd prime divisor of $t$ which, by definition of $\mathfrak{R}$, does not divide $r$. Therefore, $x \mapsto \Kl{\frac{\Delta_1}{x}}$ is not the trivial character. By following Iwaniec step by step and making use of $(n,r) ^{-1} \leq (n,N)^{-1}$ since $N \divides r$, we get 
\begin{align}  \label{eq:3.11}
F(A,B;p) \ll B^{3/2} \Kl{1+ \frac{n}{y}} (nr)^{1/2} (n,N)^{-1} \tau^2(r) \text{ log}(ny)
\end{align}
and
\begin{align}  \label{eq:3.12}
F(A,B;p) \ll A^{3/2} \Kl{1+ \frac{n}{y}} (nr)^{1/2} (n,N)^{-1} \tau^2(r) \text{ log}(ny).
\end{align}

If both $A$ and $B$ are large, we  make use of the flexibility gained through the averaging over the levels. We want to estimate  
\begin{align} \notag
F_P(A,B) = \sum_{p \in \mathcal{P}} |F(A,B;p)|. 
\end{align}
Setting $\lambda_P := \text{sgn} F(A,B;p)$ we get%\footnote{As before, the sum always runs over $a,b$ such that $(a,b)=1$ and $(ab,2nN)=1$} 
\begin{align} \notag
F_P(A,B) = \sum_{\substack{A<a\leq 2A \\ y < abr \leq 2y}} \sum_ {\substack{B<b\leq2B \\ (a,b)=1}} \sum_{\substack{P < p \leq 2P \\ p \divides ab}} \lambda_P \legendre{\pm nr}{ab} e\Kl{2n \Kl{\frac{\overline{ar}}{b}- \frac{\overline{br}}{a} + \frac{s \overline{ab}}{r} + \frac{\nu}{abr}}}.
\end{align}
To bound this, we follow \cite[Section 6]{Iw1} step by step. First, we split the sum according to whether $p \divides a$ or $p \divides b$. In each case we interchange the sums, apply Cauchy-Schwarz to the square and  change the sums back. Hence, we have two $p$-sums. If the summands of both $p$-sums coincide, we trivially estimate,  otherwise we apply the Weil bound. Since \cite[Lemma 7]{Iw1} does not hold, we cannot use $(n,r) \leq r^{1/2}$ for \cite[(6.1)]{Iw1}. Instead, we use $(n,N) \leq (n,r) \leq r$ and $(6.3)$ from Iwaniec changes to 
\begin{align}   \label{eq:40}
F_P(A,B) &\ll y r^{-1} P^{-1/2} + \Kl{1+\frac{n}{y}}^{1/2} (s^2-1,r)^{1/2} \tau(r)\, \text{log}\,y \\ \notag
&\Kl{ y^{7/8}r^{-5/8}P^{3/8} (n,N)^{-1/4} + (A^{-1/2} + B^{-1/2}) y r^{-1}}. 
\end{align}
In particular, we lose a factor of $r^{-1/4}$ in the second term within the bracket. To bound $K_Q(y)$, we modify \cite[Section 7]{Iw1} accordingly and apply \eqref{eq:3.11} and \eqref{eq:3.12} in case that either $A$ or $B$ is 
\begin{align} \notag
\leq \Kl{1+ \frac{n}{y}}^{-1/4} n^{-1/4} r^{-3/4} y^{1/2} P^{-1/2} (n,N)^{1/2}
\end{align}
respectively and \eqref{eq:40} otherwise and obtain
\begin{align}   \notag 
\sum_{p \in \mathcal{P}} \abs{F_{r,s}^{\pm}(A,B;p)} \ll &y r^{-1} P^{-1/2} + \Kl{y + n}^{5/8} r^{-5/8} (n,N)^{-1/4}
\\ \notag &(s^2-1,r)^{1/2} \tau^2(r) \,(\text{log}\,ny)  \Kl{ y^{1/4} P^{3/8} + n^{1/8}y^{1/8}P^{1/4}}.
\end{align}
According to \eqref{3.40}, it remains to sum this inequality over $s~(\text{mod}~r/2)$ and $r \in \mathfrak{R}$. The more general form of $f_r(2s,\chi)$ does not affect \cite[(7.2) and (7.3)]{Iw1}. Hence,
\begin{align} \notag
\sum_{s(\text{mod}\,r/2)} |f_r(2s,\chi)| (s^2-1,r)^{1/2}  \ll r \tau^2(r) \text{ and } \sum_{r \in \mathfrak{R}} r^{-1/8} \tau^4(r) \ll \tau(nN) N^{-1/8}. 
\end{align} 
Combining this with \eqref{eq:1}, we conclude, for quadratic $\chi$, that 
\begin{align}  \label{eq:17}
\sum_{Q \in \mathcal{Q}} \abs{K_{Q}(x)}  &\ll  \Big(  xv (n,N) n^{-1/2} N^{-1} + x P^{-1/2} N^{-1/2} \\ \notag
&+ (x+n)^{5/8} \big(x^{1/4}P^{3/8} + n^{1/8} x^{1/8} P^{1/4}\big) N^{-1/8} (n,N)^{-1/4} \Big) (nxN)^{\eps}
\end{align}
which is an improvement  of \cite[Theorem 3]{Iw1}. By \eqref{eq:008}, we infer 
\begin{align} \notag
\sum_j^d \abs{a_j(n)}^2 \ll  n^{k/2-1} \bigg(  \frac{v (n,N)}{N} + P + \frac{n^{1/2}}{P^{1/2} N^{1/2}}  + \frac{n^{3/8} P^{3/8}}{N^{1/8} (n,N)^{1/4}} \bigg) (nNP)^{\eps}.
\end{align}
Choosing $P = n^{1/7} (n,N)^{2/7} / N^{3/7} + (nN)^{\eps}$ yields, for real $\chi$, that 
\begin{align}  \notag
\sum_j^d \abs{a_j(n)}^2 \ll  n^{k/2-1}  \bigg(  \frac{v (n,N)}{N} + 1 + \frac{n^{3/7}}{(n,N)^{1/7} N^{2/7}} + \frac{n^{3/8}}{N^{1/8} (n,N)^{1/4}}  \bigg)(nN)^{\eps},
\end{align}
while, for an arbitrary character $\chi$, the first term changes to
$
\frac{v (n,N)}{N} (c_\chi \text{rad}(c_\chi))^{1/4}. 
$
This  concludes the proof for $k \geq 5$.  

To prove the case $k=3$ we follow \cite[Section 3 \& 5]{Du}, but  include an arbitrary nebentypus $\chi$. The map $f(z) \mapsto y^{3/4} f(z)$ induces an injective mapping  $S_{3/2}(N,\chi) \mapsto C_{3/4}(N,3/2,\chi)$ and one has $a(n)=(4\pi n)^{3/4} \rho(n)$, where $a(n)$ denote the Fourier coefficients of $f$ and $\rho(n)$ the  coefficients, see \eqref{eq:2.2}, of the corresponding Maa{\ss} cusp form. Let $u_i(z)$ be an orthonormal basis of Maa{\ss} cusp forms of weight $3/2$ with eigenvalues $\lambda_j$  and Fourier coefficients $\rho_j(n)$ and let $\{f_{ij}= \sum_{n \geq 1} a_{ij}(n) e(nz)\}_{i=1}^{d_j} $ be an orthonormal basis of  $S_{3/2+2j}(N,\chi)$. Then it holds, by  Proskurin's variant \cite[p. 3888]{Pr} of the Kuznetzov formula, that 
\begin{align}  \label{eq:8}
\sum_{N \divides c}  \frac{K_{\chi}^{1}(n,n;c)}{ c}  \varphi(4  \pi & n / c)   = 4n \sum_{\lambda_j >0} \frac{\abs{\rho_j(n)}^{2}}{\text{cosh}(\pi t_j)} \hat{\varphi}(t_j)  \\ \notag
 &+ \sum_{\A} \int_{-\infty}^{\infty} \frac{\abs{\phi_{\A,n}(1/2+it))}^{2}}{\text{cosh}(\pi t) \abs{\Gamma(1/2+3/4 +it}^{2}} \hat{\varphi}(t) dt  \\ \notag
&+ 4 \sum_{j \geq 1} \frac{\Gamma(3/2+2j) e(3/8+j/2) \tilde{\varphi}(3/2+2j)}{(4 \pi)^{3/2+2j} n^{1/2+2j}} \sum_{i=1}^{d_j} \abs{a_{ij}(n)}^{2}.
\end{align}
Here, $\varphi(x)$ is a suitable test function, $\sum_{\A}$ refers to the summation over the nonequivalent non-singular cusps of $\Gamma_0(N)$, $t_j$ is defined by $s_j = 1/2 + i t_j$ and $\phi_{\A,n}$ are the coefficients of an Eisenstein series (cf. \cite[p. 3876]{Pr}). Similar to the choice in \cite[p.\,51]{DS}, we set $\varphi(x) = c_0 x^{-7/2} J_{13/2}(x)$ for $c_0= - 2^{4}  e(-3/8) \pi^{-2} \Gamma(9/2)^{-1}$ and $J_k(z)$ to denote the Bessel function of order $k$. This choice fulfills all requirements for the Kuznetsov formula and by means of the Weber-Schafheitlin integral \cite[(6.574.2)]{GR} it is straightforward to calculate 
\begin{align}  \notag
\hat{\varphi}(t) =  \frac{t^2 + 1/4} { \text{cosh}(2 \pi t)\Gamma(-1/4+it) \Gamma(-1/4-it) \Gamma\Kl{6+it}\Gamma\Kl{6-it }}. 
\end{align}
Observe that  $\hat{\varphi}(t) >0$ for $t \in \R$ and for $t \in [-i/4  ,i/4]$,  the value at $it=1/4$ defined by 
\begin{align} \notag
\lim\limits_{t \to \pm i /4} \hat{\varphi}(t) =  \frac{3}{64 \pi^{3/2} \Gamma(23/4) \Gamma(25/4)}.
\end{align}
Thus, we may drop all terms of the first sum on the right-hand side of \eqref{eq:8} which represent eigenvalues distinct to $3/16$ as well as the contribution from the continuous spectrum (the integral over the Eisenstein coefficients). Since the weights of $f_{ij}$ are greater than or equal to $5/2$, we can use our previous results to bound the last term of \eqref{eq:8}. As before, we apply Iwaniec's method of averaging over the levels. If $u(z)$ is a normalized Maa{\ss} cusp form for $\Gamma_0(N)$, then $[\Gamma_0(Q):\Gamma_0(N)]^{-1/2} u(z)$ is a normalized Maa{\ss} cusp form for $\Gamma_0(Q)$, $Q \in \mathcal{Q}$. Hence, by applying the Kuznetsov formula  for every level $Q \in \mathcal{Q}$,  it follows 
\begin{align}  \label{eq:15}
n \sum_{\lambda_j=3/16} \abs{p_j(n)}^{2} \ll \text{log} P  \sum_{Q \in \mathcal{Q}}  \bigg| \sum_{Q \divides c} \frac{K_{\chi}^{1}(n,n,c)}{ c} \Kl{\frac{c}{n}}^{7/2} J_{13/2}\Kl{\frac{4 \pi n}{c}} \bigg|  \\ \notag +  \Kl{P+ \frac{n^{1/2}}{P^{1/2} N^{1/2}} + \frac{v (n,N)}{N} + \frac{n^{3/8} P^{3/8}}{N^{1/8} (n,N)^{1/4}}}  (nNP)^{\eps}.
\end{align}
Since $13/2$ is half integral and since for $x>n$ 
\begin{align} \notag
n^{-7/2} \Kl{x^3 J_{13/2}\Kl{\frac{4 \pi n}{x}}}' \ll n x^{-5/2},
\end{align}
the right-hand side of \eqref{eq:15} can be treated  exactly as in \cite[Section 8]{Iw1} taking into account \eqref{eq:17} and our choice of $P$. This concludes the proof of  Theorem \ref{main1}.

\section{An Application}

Finally, we give an application of Theorem \ref{main1}, particularly an improvement of \cite[Theorem 1.2]{Wo}. For this purpose, let $A$ be a positive, integral, symmetric $k \times k$ matrix with even diagonal elements,  let $q(x) := \frac{1}{2} x^{t} A x$ be the corresponding quadratic form and let $N$ be the level of $A$, i.e., the smallest integer such that $N A^{-1}$ is integral with even diagonal. This section aims at finding a lower bound for the Fourier coefficients $r(q,n) = \# \{x \in \Z^{k} |\, q(x)=n \}$ of $\theta(q,z)$ to conclude that $n$ is represented by $q$. By direct computation, one can show that $\theta(q,z) \in M_{k/2}(N, \chi_{(-1)^{k} \text{det}A} )$  \cite[p. 456]{Sh}. %Accordingly, $\theta(q,z)$ decomposes uniquely into an Eisenstein series and a cusp form. To study this decomposition further, one makes use of the spinor genus theory from \cite{SP}. 

Two positive quadratic forms are in the same genus if they are equivalent over all $\Z_p$. Define the theta series of the genus $\theta(\text{gen}\,q,z) = \sum_{n = 0}^{\infty}  r(\text{gen}\,q,n) e(nz)$ by 
\begin{align}  \label{eq:5.50}
r(\text{gen}\,q,n) = \sum_{\tilde{q}\in \text{gen}\,q} w(\tilde{q}) r(\tilde{q},n) \text{ with }w(\tilde{q}) = \Kl{ \sum_{\tilde{q} \in \text{gen}\,q}
\frac{1}{\# O_{\Z}(\tilde {q})}}^{-1} \frac{1}{\# O_{\Z}(\tilde {q})},
\end{align}
where the summation is taken over a set of representative classes in the genus.  Let $S(z) = \theta(q,z) - \theta(\text{gen}\,q,z)$. Then $S(z)$ is the orthogonal projection of $\theta(q,z)$ onto the subspace of cusp forms and  $\theta(\text{gen}\,q,z)$ is an Eisenstein series \cite[Korollar 1]{SP}. Consequently, write 
\begin{align} \notag
\theta(q,z) = \theta(\text{gen}\,q,z) + S(z) =: \sum_{n=0}^{\infty} r(\text{gen}\,q,n)  e(nz) + \sum_{n=1}^{\infty} a(q,n) e(nz). 
\end{align}
We would like to treat $r(\text{gen}\,q,n)$  as the main term for $r(q,n)$ and $a(q,n)$ as the error term. To compute the Eisenstein coefficients $r(\text{gen}\,q,n) $, we use  Siegel's formula \cite{Si}. From now on, let $k=3$. Then
\begin{align}  \label{eq:5.1}
r(\text{gen}\,q,n) = \frac{2\pi}{\sqrt{\Delta / 8}} n^{1/2} \prod_p r_p(q,n),
\end{align}
where $\Delta$ is  the determinant of $A$ and $r_p(q,n)$ are the $p$-adic densities defined by 
\begin{align} \notag
r_p(q,n):= \lim\limits_{\nu \to \infty} \frac{1}{p^{2\nu}} \# \{x \in (\Z/p^{\nu}\Z)^3~ |~ q(x) \equiv n ~(\text{mod\,}p^{\nu })~ \}.
\end{align}
Apart from a finite number of cases, $(p,Nn) \neq 1$, the densities are easy to compute \cite[Hilfssatz 12]{Si}    
\begin{align} \notag
r_p(q,n) = 1+ \frac{\chi_{-2n\Delta}}{p}, \quad p \nmid nN.
\end{align} 
%Note that here $4 \divides N$ since we consider half-integral weights. 
The space of theta functions $U$ poses a problem since their Fourier coefficients grow like $\asymp n^{1/2}$ which is roughly the same size as $r(\text{gen}\,q,n)$. Thus, to show that $n$ can be represented by a quadratic form $q$  using Theorem 1, it is necessary that the $n$-th coefficient of the projection of $\theta(q,z)$ onto $U$ vanishes. 

For a ring $R$ let $O_{R}(q) := \{ S \in GL_{2k}(R) | S^{t}AS = A \}$ be the finite set of $R$-automorphs of $q$. Two quadratic forms $A_1,A_2$ in the same genus with $A_1 = S^{t} A_2 S$ for  $S \in GL_{k}(\Z)$ belong to the same spinor genus, if $S \in O_{Q}(A_2) \bigcap_{p} O'_{Q_p}(A_2)GL_{k}(\Z_p)$, where $O'_{Q_p}(A)$ is the subgroup of $p$-adic automorphs $O_{Q_p}(A)$ of determinant and spinor norm 1  (cf. \cite[Section 55]{Me}). % which is exactly the commutator subgroup of $O_{Q_p}(A)$. 
Define the theta  series of the spinor genus $\theta(\text{spn}\,q,z) = \sum_{n = 0}^{\infty}  r(\text{spn}\,q,n) e(nz)$ by 
\begin{align}  \label{eq:5.60} 
r(\text{spn}\,q,n) = \sum_{\tilde{q} \in \text{spn}\,q} w(\tilde{q}) r(\tilde{q},n) \text{ with }w(\tilde{q}) = \Kl{ \sum_{\tilde{q} \in \text{spn}\,q} \frac{1}{\# O_{\Z}(\tilde {q})}}^{-1} \frac{1}{\# O_{\Z}(\tilde {q})},
\end{align}
where the summation is taken over a set of representative classes in the spinor genus of $q$.  Schulze-Pillot \cite{SP} has shown that the orthogonal projection of $\theta(q,z)$ onto the subspace of $U^{\bot}$ is $\theta(q,z)-\theta(\text{spn}\,q,z)$. Therefore,  write
\begin{align} \notag
\theta(q,z) = \theta(\text{gen}\,q,z) + H(z) + f(z), 
\end{align} 
with $H(z)= \theta(\text{spn}\,q,z) - \theta(\text{gen}\,q,z)\in U$ and $f \in  U^{\bot}$. The contribution from the Fourier coefficients of $f$ is easy to handle by Theorem \ref{main1}. If $r(\text{gen}\,q,n) = r(\text{spn}\,q,n)$, then the $n$-the Fourier coefficient of $H(z)$ vanishes. This obviously holds when  $n \notin \{ tm^2: 4t \divides N, m \in \N \}$ since  the coefficients of the theta functions vanish. According to the definitions \eqref{eq:5.50} and \eqref{eq:5.60} it follows that  $r(\text{spn}\,q,  n) = r(\text{gen}\,q, n)$ is satisfied if 
\begin{align} \notag  
r(\text{spn}\,q, n) = r(\text{spn}\,q', n)
\end{align}
for all $q'$ in the same genus as $q$. According to Schulze-Pillot \cite[Korollar 2.3 (ii)]{SP} it holds for any  $q,q'$ in the same genus and squarefree $t$ that
\begin{align}  \notag 
r(\text{spn}\,q, tm^2) = r(\text{spn}\,q', tm^2)
\end{align} 
if $N= 4 t\, t'\, h^2$ with squarefree $t'$ and $h \divides m$. In particular, if $N/4$ is squarefree, one has $\theta(\text{gen}\,q,z) = \theta(\text{spn}\,q,z)$. 

%If one would assume that the Ramanujan hypothesis is fulfilled, i.e. $a(n) \ll n^{1/4}$ for $f(z) = \sum_{n \geq 1} a(n)e(nz) \in U^{\bot}$, then one would only require the bound $n > P^{5+\delta}$. In our proof we do not need the requirement from Golubeva that $n = x^2 + y^2 + 6 P z^2 \equiv n \text{ mod }16$ is solvable and we modify the conclusion such that the solution $xyz \neq 0$. 

\begin{proof}[Proof of Theorem \ref{Golubeva}] Let $\theta(q,z)$ be the theta series of the quadratic form $q=x^2 + y^2 + 6 P z^2$. Then, $\theta(q,z) \in M_{3/2}(24P, \chi)$ for a quadratic character $\chi$ and since $6P$ is squarefree, it holds that $\theta(\text{gen}\,q,z) = \theta(\text{spn}\,q,z)$. Thus, the orthogonal projection of $\theta(q,z)$ onto the subspace of cusp forms is in $U^{\bot}$. Let $\{\varphi_j(z)= \sum_{n \geq1 } a_j(n) e(nz) \}_{j=1}^{d}$ be  an orthonormal basis of $U^{\bot}$. Then
\begin{align}  \notag
r(q,n) = r(\text{gen}\,q,n) + \sum_{j=1}^{d} c_j a_j(n) = r(\text{gen}\,q,n) + \mathcal{O}\Bigg( \sqrt{\sum_{j=1}^{d} c_{j}^{2}} \sqrt{\sum_{j=1}^{d} \abs{a_j(n)}^2}\Bigg).
\end{align}
From $\sqrt{\sum_{j=1}^{d} c_{j}^{2}} = \mathcal{O}(P^{1/4+\eps})$ (cf.  \cite[Theorem 3]{Go}) and Theorem \ref{main1},  we conclude that
\begin{align}  \label{eq:3} 
r(q,n) = r(\text{gen}\,q,n) + \mathcal{O}\Kl{v^{1/2}\Kl{ t^{13/28}P^{3/28}+t^{7/16} P^{3/16}+ t^{1/4}P^{1/4}}}(Pn)^{\epsilon}.
\end{align} 
To  bound  $r(\text{gen}\,q,n)$ from below, we apply \eqref{eq:5.1}, Siegel's formula. 
 %\begin{align}
%r(\text{gen}\,q,n) = \frac{2\pi}{\sqrt{\Delta / 8}} n^{1/2} \prod_p r_p(q,n).
%\end{align}
If $p \nmid 6P$, it holds by \cite[Hilfssatz 16]{Si} that
\begin{align} \notag
1- \frac{1}{p} \leq r_p(q,n) \leq 1 + \frac{1}{p}.
\end{align}
To treat the the remaining densities, $r_2(n,q),r_3(n,q)$ and $r_P(n,q)$, we rely on Hensel's lemma %(cf. also \cite[Section 15]{Kn}).

\begin{lemma} \label{Hensel} Assume that $P \in \Z[x_1,\ldots,x_d]$ and $\alpha \in \Z^d$  satisfy $P(\alpha) \equiv 0 \rm{~mod~} p^{k}$. If  it holds for at least one $x_j$ that
\notag
\begin{align}
\frac{\partial f}{\partial x_j}(\al) \neq 0 \rm{~mod~}p^{l} \text{ for some } l \leq \frac{k+1}{2}, 
\end{align}
then $P(x) \equiv 0 \rm{~mod~} p^{k+m}$ has $p^{m(d-1)}$ integer solutions. Each of these solutions $\be$ satisfies that $ \be_j \equiv \alpha_j \operatorname{mod} p^{k-l+1}$ and $\be_i \equiv \alpha_i \operatorname{mod} p^{k}$ \text{ for all } $i \neq j$. 
\end{lemma}
\begin{proof} The case $d=1$ is proven in \cite[p.\,48]{Ro}. Assume $j=1$. For each choice $\be_2,\ldots,\be_d$ mod $p^{k+m}$ with $\be_i \equiv \al_1$ mod $p^k$, we can apply the one-variable case  to find $\be_1$ such that  $ P(\be) \equiv 0 \text{ mod } p^{k+m}$. 
\end{proof}
For $p=2$, consider the congruence
\begin{align}  \label{eq:0.0}
x^2 + y^2 + 6 P z^2 \equiv n \text{ mod } 8
\end{align}
for arbitrary odd $n$. For each $x \equiv 1,3 \operatorname{mod}4 \ (y \equiv 1,3 \operatorname{mod}4)$, there are two possible choices for $y \operatorname{mod}8 \ (x \operatorname{mod}8)$ and four possibilities for $z \operatorname{mod} 8$ to solve \eqref{eq:0.0}. It follows by Lemma \ref{Hensel} that
\begin{align} \notag
r_2(n,q) \geq \lim\limits_{\nu \rightarrow \infty}{\frac{32 \cdot 2^{2 (\nu-3)}}{2^{2\nu} } } = 1/2. 
\end{align}
If $p$ is a prime, then $\Z \slash p\Z$ is a finite field. In a finite field of odd order $q$, every element  unequal to zero can be expressed as the sum of two squares in $q-1$ ways. Hence, for $n \nequiv 0 \text{ mod P}$, there exist $P^2 -P$ solutions of 
\begin{align}  \label{eq:0.1}
x^2 + y^2 + 6 P z^2 \equiv n \text{ mod } P,
\end{align}
with $(x,y) \nequiv 0 \text{ mod } P$. By Lemma \ref{Hensel} we infer $r_3(n,q) \geq 2/3$ and $r_P(n,q) \geq 1 - \frac{1}{P}$. It follows $r(\text{gen}\,q,n) \gg \frac{n^{1/2-\eps}} { P^{1/2}}. $ 
Thus, the main term of \eqref{eq:3} dominates the error term as soon as  
\begin{align} \notag
P \leq \text{min}(v^{14/17}t^{1/17}, v^{8/11} t^{1/11} ,v^{2/3}t^{1/3})^{1-\eps} .
\end{align}
If this holds true, it follows that $x^2+y^2+6Pz^2=n$ has a solution in $\Z^{3}$.  Furthermore, we may assume that $x,y,$ and $z$ are natural numbers since the number of integer solutions of $x^2+y^2=n$ is  $\mathcal{O} (n^{\eps})$.
\end{proof}
%To see a brief sketch how this result is related to Theorem 2, see \cite[Section 2]{Wo} and \cite[p. 544 f]{Go}. 
\begin{proof}[Proof of Theorem \ref{Wo}] We keep the notation from Wooley \cite[Section 3]{Wo} and modify only the parts concerning the bound of Golubeva's theorem. The necessary requirements to apply Theorem \ref{Golubeva}, $(i)\, NM^{12} > p^{17}, (ii)\,NM^6 > p^{11}$ and $(iii)\, N>p^{3}$, are fulfilled provided that (cf. \cite[p.\,14]{Wo})
\begin{enumerate}[label=$(\roman*)$]
\item  $\gamma_0(6/c+1) - 4/c - \eps > 17 \gamma_0 - 34/3 + \eps,$
\item  $\gamma_0(3/c+1) - 2/c - \eps > 11 \gamma_0 - 22/3 + \eps$ and
\item  $\gamma_0 - \eps > 3 \gamma_0 - 2 + \eps. $
\end{enumerate}
These inequalities yield the following conditions
\begin{align} \notag
(i) ~\gamma_0 < \frac{34 c -12  -6 c \eps}{48c-18}, ~ (ii)~ \gamma_0 < \frac{22c - 6 -6 c \eps}{30c-9} \text{ and } (iii)~ \gamma_0 <1 - 2 \eps. 
\end{align}
Assuming the Riemann hypothesis, Wooley chooses $c=2+2\eps$ (cf. \cite[p.15]{Wo}). With this choice and $\eps$ sufficiently small, the conditions are satisfied as long as $\gamma_0 < 28/39= \rm{min}(28/39,38/51,1)$.  Otherwise, without assuming the Riemann hypothesis, the choice is $c=\frac{12}{5}+2 \eps$, and it follows $\gamma_0 < 58/81 = \rm{min}(58/81,26/35,1)$. The rest of the proof can be conducted exactly as in \cite[Section 3]{Wo}. 
\end{proof}
\bibliographystyle{amsplain}
\bibliography{mybib}

\providecommand{\bysame}{\leavevmode\hbox to3em{\hrulefill}\thinspace}
\providecommand{\MR}{\relax\ifhmode\unskip\space\fi MR }
\providecommand{\MRhref}[2]{%
  \href{http://www.ams.org/mathscinet-getitem?mr=#1}{#2}
}
\providecommand{\href}[2]{#2}
\begin{thebibliography}{10}

\bibitem{Bl1}
V.~Blomer, \emph{Uniform bounds for fourier coefficients of theta-series with
  arithmetic applications}, Acta Arith. \textbf{114} (2004), 1--21.

\bibitem{Du}
W.~Duke, \emph{Hyperbolic distribution problems and half-integral weight maass
  forms}, Invent. Math. \textbf{92} (1988), 73--90.

\bibitem{DS}
W.~Duke and R.~Schulze-Pillot, \emph{Representation of integers by positive
  ternary quadratic forms and equidistribution of lattice points on
  ellipsoids}, Invent. Math. \textbf{99} (1990), 49--57.

\bibitem{Go}
E.P. Golubeva, \emph{A bound for the representability of large numbers by
  ternary quadratic forms and nonhomogenous waring equations}, J. Math. Sci.
  \textbf{157} (2009), no.~4, 543--552.

\bibitem{GR}
I.~S. Gradshteyn and I.~M. Ryzhik, \emph{Table of integrals, series, and
  products}, 5 ed., Academic Press, 1994.

\bibitem{Iw1}
H.~Iwaniec, \emph{Fourier coefficients of modular forms of half-integral
  weight}, Invent. Math. \textbf{87} (1987), 385--401.

\bibitem{Iw2}
H.~Iwaniec and E.~Kowalski, \emph{Analytic number theory}, Colloq. Publ.,
  vol.~53, Amer. Math. Soc., 2004.

\bibitem{KL}
A.~Knightly and C.~Li, \emph{Kuznetsov's trace formula and the hecke
  eigenvalues of maass forms}, Mem., vol. 224, Amer. Math. Soc., 2013.

\bibitem{Ko}
N.~Koblitz, \emph{Introduction to elliptic curves and modular forms}, Grad.
  Texts in Math., vol.~97, Springer, 1993.

\bibitem{Ob}
M.~Oberhettinger, \emph{Formulas and theorems for the special functions of
  mathematical physics}, Springer, 1966.

\bibitem{Me}
O.T. O'Meara, \emph{Introduction to quadratic forms}, Grundlehren Math. Wiss.,
  vol. 117, Springer, New York, 1973.

\bibitem{Pr}
N.V. Proskurin, \emph{On the general kloosterman sums}, J. Math. Sci.
  \textbf{129} (2005), no.~3, 3874--3889.

\bibitem{Ro}
A.M. Robert, \emph{A course in p-adic analysis}, Grad. Texts in Math., vol.
  198, Springer, 2000.

\bibitem{Sar}
P.~Sarnak, \emph{Some applications of modular forms}, Cambridge Tracts in
  Math., vol.~99, Cambridge University Press, 1990.

\bibitem{SP}
R.~Schulze-Pillot, \emph{Thetareihen positiv definiter quadratischer formen},
  Invent. Math. \textbf{75} (1984), 283--299.

\bibitem{Se}
A.~Selberg, \emph{Harmonic analysis and discontinuous groups in weakly
  symmetric riemannian spaces with applications to dirichlet series}, J. Indian
  Math. Soc. \textbf{20} (1956), 47--87.

\bibitem{Sh}
G.~Shimura, \emph{On modular forms of half integral weight}, Ann. of Math. (2)
  \textbf{97} (1973), 440--481.

\bibitem{Si}
C.L. Siegel, \emph{{\"U}ber die analytische theorie der quadratischen formen},
  Ann. of Math. (2) \textbf{38} (1935), 212--291.

\bibitem{Wo}
T.D. Wooley, \emph{On waring's problem: some consequences of golubeva's
  method}, J. London Math. Soc. \textbf{88} (2013), 699--715.

\end{thebibliography}
%\nocite{•}
%\printbibliography
\end{document}